\documentclass[11pt]{amsart}
\usepackage[english]{babel}
\usepackage{a4wide,color,graphicx,amsmath,amssymb,verbatim,mathrsfs,latexsym}
 \usepackage[colorlinks=true]{hyperref}
\allowdisplaybreaks
\newtheorem{theo}{Theorem}
\newtheorem{prop}{Proposition}
\newtheorem{lemma}{Lemma}

\newcommand{\R}{\mathbb{R}}	
\newcommand{\Div}{\textrm{div}\,}	

\newcommand{\na}{\nabla}		
	
\def\dive{\textnormal{div}}

\title{Uniqueness of higher integrable solution to the Landau equation with Coulomb interactions}
\author{Jann-Long Chern and Maria Pia Gualdani}
\address{Department of Mathematics, The University of Texas at Austin,  Austin TX, USA. }
\email{gualdani@math.utexas.edu}

\address{Department of Mathematics, National Taiwan Normal University, Taipei 11677, Taiwan.}
\email{chern@math.ntnu.edu.tw}
\date{}

\begin{document}

\thanks{JLC is support by MOST 110-2115-M-003-019-MY3 and MOST 111-2218-E-008-004-MBK. 
MPG is supported by NSF DMS-2019335. JLC and MPG would like to thank KTH Royal Institute of Technology and then NCTS Mathematics Division Taipei for their kind hospitality.} 

\begin{abstract}
We are concerned with the uniqueness of weak solution to the spatially homogeneous Landau equation with Coulomb interactions under the assumption that the solution is bounded in the space $L^\infty(0,T,L^p(\R^3))$ for some $p>3/2$. The proof uses a weighted Poincar\'e-Sobolev inequality recently introduced in \cite{GG18}.
\end{abstract}

\maketitle

\pagestyle{headings}		

\markboth{Uniqueness of smooth solutions to Landau-Coulomb}{J.L. Chern, M. Gualdani}

\section{Introduction} 
The Landau equation was introduced in $1936$ by Lev Landau as a correction of the Boltzmann equation to describe collision of particles interacting under a potential of Coulomb type. Collisions of such kind are predominant in hot plasma. In its homogeneous form  the Landau equation reads as 
\begin{align}\label{eqn:Landau}
  \partial_tf =Q(f,f),
\end{align}
where $f = f(v,t)$  for $v\in \R^3,\; t>0$ is a nonnegative function describing the evolution of the particle density and 
\begin{align}\label{eqn:Landau Collisional Term Classical Expression}
  Q(f,f) :=  \frac{1}{8\pi}  \dive\left (\int_{\mathbb{R}^3} \frac{1}{|v-w|}\left (\Pi(v-w)(f(w)\nabla_vf(v)-f(v)\nabla_wf(w) \right )\;dw \right),
  \end{align}
  with $\Pi(z)$ the projection onto the orthogonal subspace of $z$,
  \begin{align*}
  \Pi(z) := \mathbb{Id}- \frac{ z\otimes z}{|z|^{2}}, \quad z\neq 0.
  \end{align*}
  Equation (\ref{eqn:Landau})-(\ref{eqn:Landau Collisional Term Classical Expression}) has been extensively studied in the literature but the main question whether or not after a certain time solutions could become unbounded  is still open. The possible blow-up in the $L^\infty$-norm could be caused by the quadratic nonlinearity in (\ref{eqn:Landau Collisional Term Classical Expression}): assuming that $f$ is smooth enough, one can rewrite (\ref{eqn:Landau Collisional Term Classical Expression}) as 
$$
 Q(f,f) =  \Div(A[f]\na f - f\na a[f]),
$$
where $A[f]$ is the diffusion matrix defined as 
$$
A[f](v,t)= \{a_{i,j}\}_{i,j} := \frac{1}{8\pi} \int_{\R^3}  \frac{1}{|w|} \left(  \mathbb{Id} - \frac{w \otimes w}{|w|^2} \right) f(v-w,t)\;dw,
$$
and 
$$
a[f](v,t) := \textrm{tr}(A[f]) = (-\Delta)^{-1} f,
$$
or in non-divergence form
\begin{align*}
  Q(f,f) = \textrm{tr}(A[f]D^2(f))+f^2.
\end{align*}
In the last formulation the quadratic nonlinearity is explicit. 

Before we state the main result of this manuscript we briefly review the literature for (\ref{eqn:Landau})-(\ref{eqn:Landau Collisional Term Classical Expression}), omitting the rather large literature on non-Coulomb potentials and spatially inhomogeneous case. Existing literature for (\ref{eqn:Landau})-(\ref{eqn:Landau Collisional Term Classical Expression}) includes results on (i) local in time well-posedness of solutions, (ii) global in time existence and uniqueness of smooth solution for initial data close to equilibrium \cite{G02}, (iii) global in time existence of (very) weak solutions \cite{AleVil2004, Des2015,Vil1998} , and (iv) convergence of weak solutions towards the equilibrium function (Maxwellian) in the $L^1$-norm \cite{CDH17}.  Very recently the second author and collaborators studied the partial regularity of weak solutions to (\ref{eqn:Landau})-(\ref{eqn:Landau Collisional Term Classical Expression}) and showed in  \cite{GGIV19} that the Hausdorff measure of the set of singular times (i.e. times at which the function could be unbounded) is at most $\frac{1}{2}$.  We also mention an important result from \cite{GG}; there the authors study an {\em{isotropic}} version of the Landau equation, previously introduced by Krieger and Strain in \cite{KS},
\begin{align}\label{eqn:Krieger-Strain}
  \partial_t f = \Div(a[f]\na f - f\na a[f]),
\end{align}	
and show that (\ref{eqn:Krieger-Strain}) with spherically symmetric and radially decreasing initial data (but not small neither near equilibrium!) has smooth solutions which remain bounded for all times.

Since the main question of global well-posedness for general initial data  for (\ref{eqn:Landau})-(\ref{eqn:Landau Collisional Term Classical Expression}) is still open, in the most recent years there have been several conditional proofs of existence of bounded solutions and their regularity. In this directions we mention \cite{L17, GG, GG18, GIMV17}.

In the current manuscript we are concerned with uniqueness of weak solutions in the class of higher integrable solutions, namely we assume that weak solutions belong to $L^\infty(0,T,L^p(\R^3))$ for some $p>\frac{3}{2}$ and have high enough bounded moments.  Conditional uniqueness of bounded weak solutions for Landau-Coulomb has been previously studied in \cite{F10}; via a probabilistic approach using a stochastic representation of (\ref{eqn:Landau})-(\ref{eqn:Landau Collisional Term Classical Expression}) the author shows uniqueness in the class of solutions $L^1(0,T,L^\infty(\mathbb{R}^3))$. A similar approach was recently used in \cite{SW19} for the relativistic Landau-Coulomb equation.

Here is our main result:

\begin{theo}\label{main_thm_1}
The homogeneous Landau-Coulomb equation with initial data such that
\begin{align}\label{cond_init_data}
f_{in} \ge 0, \quad  \int_{\R^3} f^2_{in} (1+|v|)^{{5}}\;dv \le C, \quad \int_{\R^3} f_{in} (1+|v|)^q\;dv \le C,
\end{align}
for $q =  \frac{{{46}}(p-1)}{p-3/2}$,  has at most one solution in the time interval $[0,T]$, $T>0$, in the class of functions  
\begin{align}
f\in L^\infty(0,T,L^p(\R^3)), \quad {{p >3/2}}.
\end{align}
\end{theo}

The proof of Theorem \ref{main_thm_1} differs from the one in \cite{F10} in several aspects. We only require our solution to belong to some $L^p(\mathbb{R}^3)$ space with $p>3/2$, uniformly in time. Our method uses the weak representation of  (\ref{eqn:Landau})-(\ref{eqn:Landau Collisional Term Classical Expression}) provided in \cite{GZ} and a new weighted Poincar\'e inequality (\ref{eqn:epsilon_Poincare_inequality strongerII_bounded_II_ineq}) recently introduced in \cite{GG18}. This inequality is shown to be valid for any solution $f$ to the Landau equation that is uniformly in time $L^p(\R^3)$-integrable, for some $p>3/2$ (and has high enough bounded moments). The question whether (\ref{eqn:epsilon_Poincare_inequality strongerII_bounded_II_ineq}) holds without the extra integrability assumption is still open and very interesting. In \cite{GG18} the authors showed that (\ref{eqn:epsilon_Poincare_inequality strongerII_bounded_II_ineq}) {\em nearly} holds if we only assume uniformly in time $L^1(\R^3)$- integrability for $f$; this means that the diffusion $\textrm{div}(A[f]\nabla f)$ and the reaction $f^2$ are of the same order. In this regard we should think of (\ref{eqn:Landau})-(\ref{eqn:Landau Collisional Term Classical Expression}) as a {\em critical} equation.  

The rest of the manuscript is organized as follows: in Section \ref{old_results} we recall some useful well-known results, in Section \ref{sec:eps_poinc} we present the weighted Poincar\'e  inequality. In Section \ref{sec:weighted_estimates} we show integrability and  weighted estimates for the gradient. Section \ref{sec:contraction_argument} contains the proof of Theorem \ref{main_thm_1}.

\section{Well-known results}\label{old_results}
The following quantities will be frequently used throughout the paper. We respectively define the mass, momentum and entropy of a nonnegative function $h(v)$ the quantities 
\begin{align*}
\int_{\R^3} h(v)\;dv, \quad \int_{\R^3} h(v)|v|^2\;dv, \quad \int_{\R^3} h(v)\ln h(v)\;dv.  
\end{align*}
We start by recalling the definition of weak solution \cite{Des2015}: given initial data $f_{in}$ with finite mass, first, second moment and entropy, a weak solution to the Landau  is a nonnegative function $f$ such that  $(1+|v|^2)^{-3/2} f\in L^1(0,T,L^3(\mathbb{R}^3))$, has finite mass, first, second momentum and entropy and for all $\varphi \in C^2_c([0,T]\times \mathbb{R}^3)$
\begin{align}\label{weak_sol_I}
-\int_{\R^3}& f_{in}(v)\varphi(v,0)\;dv - \int_0^T \int_{\R^3}f(v,t) \partial_t \varphi(v,t)\;dvdt \nonumber \\
& = \frac{1}{2} \sum_{i=1}^{3}\sum_{j=1}^{3}\int_0^T \int_{\R^3} \int_{\R^3} f(v,t) f(w,t) a_{ij}(v-w)\left(\partial_{ij}\varphi(v,t)+ \partial_{ij}\varphi(w,t)\right)\;dvdwdt \\
&+ \sum_{i=1}^{3} \int_0^T \int_{\R^3} \int_{\R^3}  f(v,t) f(w,t) (\textrm{div}_v A[f])_{i}(v-w)\left(\partial_{i}\varphi(v,t)- \partial_{j}\varphi(w,t)\right)\;dvdwdt.\nonumber
\end{align}
Recently the authors in \cite{GZ} improved the regularity of the weak solutions: let $f$ be a weak solution to the Landau equation as in (\ref{weak_sol_I}); then $A[f]\in L^{\infty}(0,T; L^{3}_{loc}(\R^{3}))$, $\nabla a[f] \in L^{\infty}(0,T; L^{3/2}_{loc}(\R^{3}))$, and  for all $\phi\in L^{\infty}(0,T; W^{1,\infty}_{c}(\R^{3}))$ the function $f$ satisfies 
\begin{align}\label{weak_sol}
&\int_0^T\langle \partial_{t}f\, , \phi\,\rangle dt + \int_0^T\int_{\R^3}(A[f]\nabla f - f\nabla a[f])\cdot\nabla\phi\, dv dt = 0.
\end{align}

Next we recall some well-known results used later in the manuscript. The first one concerns lower bounds for $a[f]$ and $A[f]$. 
\begin{lemma} (Bound from below) \label{lem: A lower bound in terms of conserved quantities}
  There is a constant $c$ only determined by the mass, energy, and entropy of $f$, such that for all $v\in \mathbb{R}^3$ 
  \begin{align*}
  a[f](v) &\geq c \langle v\rangle^{-1}, \\
    A[f](v) &\ge a^*(v)\mathbb{I} \geq c \langle v\rangle^{-3}\mathbb{I},  
  \end{align*}	 
  where $ \langle v\rangle:=(1+|v|^2)^{1/2}$ and $a^*(v)$ is the smallest eigenvalue of $A[f]$ defined as 
  $$
a^*(v)=  \inf_{e\in \mathbb{S}^2} (A[f](v)e,e).
  $$
  \end{lemma}

   \begin{lemma} (Propagation of moments, \cite{Des2015} Proposition 4.) \label{lem: propagation of moments} Let $f$ be a weak solution to the Landau equation with initial datum $f_{in}$. Assume also that $f$ satisfies the conservation of mass, momentum and energy.  For all $k \ge 0$ such that 
   $$
   \int_{\R^3} f_{in}(1+|v|^2)^k \;dv < +\infty,$$
   we have that 
    $$
   \sup_{[0,T]}\int_{\R^3} f(1+|v|^2)^k \;dv\le C(1+T),$$
   where $C$ depends on the energy, mass, entropy and $k$-moments of the initial data.

\end{lemma}
We also recall the Boltzmann H-Theorem: let $\rho_{in}$ denote the Maxwellian with same mass, center of mass, and energy as $f_{in}$. We have
  \begin{align}\label{eqn:entropy production bound}
    \int_{0}^{T}\int_{\mathbb{R}^3}4(A[f]\nabla f^{1/2},\nabla f^{1/2})-f^2\;dvdt \leq H(f_{in}) - H(\rho_{{in}}).
  \end{align}	

\section{The $\varepsilon$-Poincar\'e inequality} \label{sec:eps_poinc}

In this section we present a weighted Poincar\'e inequality; this inequality plays a key role in the proof of Theorem \ref{main_thm_1}.  It is an adaptation of another inequality proven in \cite{GG18} and is based on the general weighted Poincar\'e-Sobolev inequality shown in \cite{SW}.
\begin{theo}\label{theo3_bis}
Let $N\ge 1$ and assume there exist $s>1$, a nonnegative function $f$ and a modulus of continuity $\eta(\cdot)$ such that for any cube $Q \subset \R^3$ with side length $r\in(0,1)$ the following inequality holds: 
 \begin{align}\label{eqn:epsilon_Poincare_inequality strongerII_bounded_II}
    |Q|^{\frac{1}{3}} \left( \frac{1}{|Q|}\int_{Q} (1+|v|)^{Ns/2 }f^s\;dv\right)^{\frac{1}{2s}}\left(\frac{1}{|Q|}\int_{Q} (1+|v|)^{3s}\;dv\right)^{\frac{1}{2s}}  \leq \eta(r).
  \end{align}	
  Then, given any $\varepsilon \in (0,1)$, for any smooth functions $\phi$ we have the the following $\varepsilon$-Poincar\'e inequality:
    \begin{align}\label{eqn:epsilon_Poincare_inequality strongerII_bounded_II_ineq}
    \int_{\mathbb{R}^3}  (1+|v|)^{N/2 }f \phi^2\;dv \leq \varepsilon \int_{\mathbb{R}^3}  (1+|v|)^{-3}|\nabla\phi|^2 \;dv + \tilde
   \eta(\varepsilon)\int_{\mathbb{R}^3}	\phi^2\;dv,
  \end{align} 
  where $\tilde \eta:(0,1)\mapsto\mathbb{R}$ is a decreasing function with $\tilde\eta(0+)=\infty$ determined by $\eta$.
\end{theo}
\begin{proof}
The proof can be found in Theorem 2.7 in \cite{GG18}.
\end{proof}



The validity of (\ref{eqn:epsilon_Poincare_inequality strongerII_bounded_II_ineq}) depends on certain properties of the function $f$; most importantly, the value $\varepsilon$ depends on the modulus of continuity $\eta(\cdot)$ in (\ref{eqn:epsilon_Poincare_inequality strongerII_bounded_II}). The next proposition shows that (\ref{eqn:epsilon_Poincare_inequality strongerII_bounded_II}) is satisfied if $f\in L^\infty(0,T,L^p\cap L^1(\R^3))$ for some $p > \frac{3}{2}$ and has high enough moments.
\begin{prop}\label{prop1}
Let $f$ be a nonnegative function with $f\in L^\infty(0,T,L^p\cap L^1(\R^3))$ for some $p > \frac{3}{2}$.  Assume also that $f$ has bounded moments of order $\frac{(N+6)(p-1)}{p-3/2}$.
Then there exists a number $s\le 2$ with $\frac{3}{2}<s<p$ and a modulus of continuity $\eta(r)$ such that for any $Q$ cube in $\mathbb{R}^3$ with length  $r$ inequality (\ref{eqn:epsilon_Poincare_inequality strongerII_bounded_II}) holds.
\end{prop}

\begin{proof}
Let $Q$ a cube of length $r$ and center $v_0$. H\"older inequality yields 
\begin{align*}
\int_{Q} (1+|v|)^{N s/2}f^s\;dv &\le \left( \int_{Q} (1+|v|)^{N s \alpha/2}f\;dv \right)^{\frac{1}{\alpha}} \left(\int_{Q} f^{(s-1/\alpha)\alpha'}\;dv \right)^{\frac{1}{\alpha'}}\\
& \le \|f\|^{\alpha'p}_{L^p} \left( \int_{Q} (1+|v|)^{N s \alpha/2}f\;dv \right)^{\frac{1}{\alpha}},
\end{align*}
by choosing $\alpha= \frac{p-1}{p-s}$, $p>s$, so that 
$$
(s-1/\alpha)\alpha' =p.
$$
Then
\begin{align*}
    |Q|^{\frac{1}{3}} \left( \frac{1}{|Q|}\int_{Q} (1+|v|)^{N s/2}f^s\;dv\right)^{\frac{1}{2s}} & \left(\frac{1}{|Q|}\int_{Q} (1+|v|)^{3s}\;dv\right)^{\frac{1}{2s}} \\
  &   \leq C(\|f\|_{L^p})|Q|^{\frac{1}{3} - \frac{1}{2s}} (1+|v_0|)^{3/2}\left( \int_{Q} (1+|v|)^{N s \alpha/2}f\;dv \right)^{\frac{1}{2s\alpha}}\\
  &   \le C(\|f\|_{L^p}) |Q|^{\frac{1}{3} - \frac{1}{2s}} \left( \int_{Q} (1+|v|)^{(N+6) s \alpha/2}f\;dv \right)^{\frac{1}{2s\alpha}}\\
  &   \le C(\|f\|_{L^p})  |Q|^{\frac{1}{3} - \frac{1}{2s}} \|f \langle v \rangle^{(N+6)s\alpha/2}\|_{L^1}^{\frac{1}{2s\alpha}}.
  \end{align*}	
The modulus of continuity $\eta(r)$ is proportional to $C(T)r^{1-\frac{3}{2s}}$, where $C(T)$ depends on the $\frac{(N+6)s\alpha}{2}$-moments of $f$ and on the $L^p$ norm of $f$.

\end{proof}

\section{Higher integrability and weighted gradient estimates} \label{sec:weighted_estimates}

The first immediate consequence of Proposition \ref{prop1}, Theorem \ref{theo3_bis} and Boltzmann's H Theorem is a $L^2((0,T),L^2(\mathbb{R}^3))$ integrability estimate for $f$.

\begin{theo}\label{theo3}
Let $f$ be a solution to the Landau equation with initial datum $f_{in}$ such that $f\in L^\infty((0,T),L^p(\mathbb{R}^3))$ for some $p>3/2$. Assume moreover that 
$$
\int_{\mathbb{R}^3} f_{in}(1+|v|)^{k}\;dv <+\infty,
$$
for any $1\le k\le\frac{6(p-1)}{p-3/2}$. Then $f\in L^2((0,T),L^2(\mathbb{R}^3))$ and 
\begin{align*}
\|f\|_{L^2((0,T),L^2(\mathbb{R}^3))} \le C(f_{in},T,\left\| f \right\|_{L^\infty (L^p)}).
\end{align*}
\end{theo}
\begin{proof}
The function $f$ satisfies the assumptions for (\ref{eqn:epsilon_Poincare_inequality strongerII_bounded_II}), following Proposition \ref{prop1}. Then, combining  (\ref{eqn:epsilon_Poincare_inequality strongerII_bounded_II_ineq}) with $\phi = \sqrt{f}$, $N=0$ and (\ref{eqn:entropy production bound}),  we get:
 \begin{align*}
    \int_0^T \int_{\mathbb{R}^3} f^2\;dvdt \leq \varepsilon \int_0^T \int_{\mathbb{R}^3} (A{[f]}\nabla \sqrt{f},\nabla \sqrt{f})\;dvdt + C(f_{in})\tilde
   \eta(\varepsilon)T \\
   \le \varepsilon \left(H(f_{in}) - H (\rho_{f_{in}}) + \int_0^T \int_{\mathbb{R}^3} f^2\;dvdt\right) + C(f_{in})\tilde
   \eta(\varepsilon)T .
  \end{align*} 
  The thesis follows by choosing $\varepsilon <1$.

   \end{proof}

Once we have the bound $L^2((0,T),L^2(\mathbb{R}^3))$, we can get an estimate for $f$ in the space $L^\infty((0,T),L^2(\mathbb{R}^3))$, as shown in the following theorem: 
\begin{theo}\label{theo4}
Let $f$ and $f_{in}$ as is Theorem \ref{theo3}.  Assume moreover that $f_{in} \in L^2(\mathbb{R}^3)$. Then $f\in L^\infty((0,T),L^2(\mathbb{R}^3))$ and $\left\| f \right\|_{L^\infty(L^2)} \le C(f_{in},T, \left\| f \right\|_{L^\infty(L^p)}).$
\end{theo}
\begin{proof}
The proof is a simple consequence of Gronwall's lemma. Take $f$ as test function in (\ref{weak_sol}) and integrate by parts; this gives 
\begin{align}\label{ma_non_lo_so}
 \int_{\mathbb{R}^3} f^2(T)\;dv =  \int_{\mathbb{R}^3} f^2_{in} \;dv - \int_0^T \int_{\mathbb{R}^3} (A{[f]}\nabla {f},\nabla {f})\;dvdt + \int_0^T \int_{\mathbb{R}^3} f^3\;dvdt.
\end{align}
Since $f\in L^2((0,T),L^2(\mathbb{R}^3))$ by Theorem \ref{theo3}, we use  (\ref{eqn:epsilon_Poincare_inequality strongerII_bounded_II_ineq}) with $\phi = f$, $N=0$,  and $\varepsilon <1$ and get 
\begin{align*}
 \int_{\mathbb{R}^3} f^2(T)\;dv \le   \int_{\mathbb{R}^3} f^2_{in} \;dv  + \frac{1}{\varepsilon} \int_0^T \int_{\mathbb{R}^3} f^2\;dvdt.
\end{align*}
Gronwall's inequality yields
\begin{align*}
 \int_{\mathbb{R}^3} f^2(T)\;dv \le e^{\frac{1}{\varepsilon}T}\int_{\mathbb{R}^3} f^2_{in} \;dv .
\end{align*}
Note that the above computations are formal. To make them rigorous one first considers a truncation of $f$ of the form $f\eta_R(v)$ where $\eta_R(v) = \eta(v/R) $ and $\eta(v)=1$ inside a ball of center $0$ and radius $1$, $\eta(v)=0$ outside the ball of center $0$ and radius $2$ and smooth in between. Thanks to the condition that $f\in L^\infty((0,T),L^p\cap L^1(\mathbb{R}^3))$ for some $p>3/2$ both $A{[f]}$ and $a{[f]}$ are uniformly bounded and one can take $f\eta_R(v)$ as test function in (\ref{weak_sol}). Since $\nabla \eta_R \to 0$ as $R \to +\infty$ and both $A{[f]}$ and $a{[f]}$ are uniformly bounded one can pass to the limit  $R \to +\infty$ and obtain (\ref{ma_non_lo_so}).

\end{proof}

{

For proving our uniqueness result, we also need the following weighted gradient bound.
%
%
%
%
%
%
\begin{prop}\label{Prop_Weighted_Gradient_Estimate} 
Let $N\ge 0$ and $f\in L^\infty(0,T,L^p)$ with {{$p>3/2$}} be a weak solution to the Landau equation with initial data $f_{in} \in L^1 \cap L^2 (\mathbb{R}^3)$ and $\frac{(4N+6)(p-1)}{p-3/2}$-moments bounded. Let moreover $ \int_{\R^3} f_{in}^2(1+|v|)^N \;dv < +\infty$. For any $T>0$ we have 
\begin{align*}
\int_{\R^3} f^2(1+|v|)^N \;dv + \frac{1}{2} \int_0^T \int_{\R^3} (1+|v|)^{N-3}|\nabla f|^2 \;dvdt \le C(T,f_{in},\|f\|_{L^\infty(0,T,L^p)}).\\
\end{align*}
\end{prop}
\begin{proof}
There exists a universal constant $C$ such that
\begin{align}\label{a_inf_gen}
 \|a[f]\|_{L^\infty(\R^3)} \le C \|f\|^{1-q/3}_{L^{1}}\|f\|^{q/3}_{L^{q/(q-1)}},\quad \forall \; 1\le q <3.
\end{align}
In particular
\begin{align}\label{a_infty_norm_est}
 \|a[f]\|_{L^\infty(\R^3)} \le  C \|f\|^{1/3}_{L^{1}}\|f\|^{2/3}_{L^{2}} \le C\|f \langle v \rangle^{m/2} \|_{L^{2}},
\end{align}
for $m>3$ and $\langle v \rangle := (1+ |v|^2)^{1/2}$.  
For large $v$, one can obtain a sharper estimate: 
\begin{align}\label{a_infty_est}
 a[f](v,t) \le \frac{C(\|f\|_{L^{3/2^+}}, f_{in}, T)}{1+|v|}, \quad \forall v\in \mathbb{R}^3\; t\in [0,T].
\end{align}
Let $|v|$ be large enough;  for $2\ge s>3/2$ H\"older inequality yields:
\begin{align*}
|a[f]|  \le & |v|^{\frac{3-s'}{s'}}\left(\int_{B_{\frac{|v|}{2}}(|v|)} f^s\;dy\right)^{1/s}+ \frac{1}{|v|} \|f\|_{L^1(\R^3)} \\
 \le & c \frac{|v|^{\frac{3-s'}{s'}}}{(1+|v|)^{\lambda/s}}\left(\int_{\R^3} f^s(1+|y|)^{\lambda}\;dy\right)^{1/s}+ \frac{1}{|v|} \|f\|_{L^1(\R^3)},
\end{align*}
 with $\frac{1}{s}+\frac{1}{s'}=1$ and $s'<3$. Chose $\lambda={3}{(s-1)}$ so that $\frac{3-s'}{s'} - \frac{\lambda}{s}=-1$ and get 
\begin{align*}
|a[f](v)|  \le\frac{1}{(1+|v|)}\left(\int_{\R^3} f^s(1+|y|)^{{3}{(s-1)}}\;dy\right)^{1/s}+ \frac{1}{|v|} \|f\|_{L^1(\R^3)}.
\end{align*}
H\"older's inequality yields
\begin{align*}
\int_{\R^3} f^s(1+|y|)^{{3}{(s-1)}}\;dy \le \left(\int_{\R^3} f^pdy\right)^{1/\alpha'}\left(\int_{\R^3} f (1+|y|)^{\frac{3(p-1)}{(p-3/2)}}dy\right)^{1/\alpha},
\end{align*}
with $\alpha = (p-1)/(p-3/2)$. We use Lemma \ref{lem: propagation of moments} to bound the last integral and get (\ref{a_infty_est}).

Take now $\phi:=f(1+|v|)^N$ as test function in (\ref{weak_sol}):
\begin{align*}
\int_{\R^3} f_t f(1+|v|)^N\;dv 
\le &\; -\int_{\R^3} \langle A[f](1+|v|)^N\nabla f, \nabla f\rangle \;dv + N\int_{\R^3} a[f] (1+|v|)^{N-1} f |\nabla  f| \;dv  \\
&+ \int_{\R^3} f (1+|v|)^N \nabla f \cdot \nabla a[f]\;dv + N \int_{\R^3} f^2 (1+|v|)^{N-1} | \nabla a[f]|\;dv\\
= &\; I_1+I_2+I_3+I_4. 
\end{align*}
%
%
%
%
%
%
By Lemma \ref{lem: A lower bound in terms of conserved quantities} we have
$$
I_1 \le -c_1 \int_{\R^3} (1+|v|)^{N-3}|\nabla f|^2 \;dv. 
$$
Using (\ref{a_inf_gen}) to bound the $L^\infty$-norm of $a[f]$, Young's inequality yields 
%
%
%
%
%
%
\begin{align*}
I_2  & \le \omega \int_{\R^3} (1+|v|)^{N-3}|\nabla f|^2 \;dv  + \frac{C^2}{\omega} \int_{\R^3} (1+|v|)^{N-2} f^2 \;dv,
\end{align*}
where $C$ only depends on the $L^\infty(0,T,L^p)$ and on the $L^\infty(0,T,L^1)$-norm of $f$. 
%
%
%
%
%
%
{\noindent
From integration by parts one obtains
\begin{align*}
I_3 + I_4  \le &\;  c_1 \int_{\R^3} (1+|v|)^{N-1} f^2|\nabla a[f]| \;dv + c_{2}\int_{\R^3} (1+|v|)^N f^3\;dv\\
\lesssim &\; \int_{\R^3}|\nabla a[f]|^3\;dv +  \int_{\R^3} (1+|v|)^{2N} f^3 \;dv \\
\le & \; C\|f\|^3_{L^{3/2}} +  \varepsilon \int_{\mathbb{R}^3}  (1+|v|)^{-3}|\nabla f|^2 \;dv + \tilde
   \eta(\varepsilon)\int_{\mathbb{R}^3}	f^2\;dv,
\end{align*}
using Hardy-Littlewood-Sobolev inequality 
\begin{align}\label{nabla_a_L3}
\|\nabla a[f]\|_{L^{3p/(3-p)}(\R^3)} \le C\|f\|_{L^p(\R^3)} \;\; \forall \; p \in (1, 2],
\end{align}
and  (\ref{eqn:epsilon_Poincare_inequality strongerII_bounded_II_ineq}) with weight $(1+|v|)^{2N}  f$ to bound the weighted cubic norm of $f$. 
}
Summarizing we have 
\begin{align*}
\partial_t \int_{\R^3}  f^2(1+|v|)^N\;dv \le & -(c -\varepsilon)\int_{\R^3} (1+|v|)^{N-3}|\nabla f|^2 \;dv \\
& + C \int_{\R^3} (1+|v|)^{N-2} f^2 \;dv +  C(\|f\|_{L^p}) .
\end{align*}
Taking $\varepsilon$ sufficiently small we get the desired estimate. 
\end{proof}

\section{The contraction argument}\label{sec:contraction_argument}
We have the following uniqueness result.
}
\begin{theo}
Let $u, \phi \in L^\infty(0,T,L^p)$ for some $p>3/2$ be two solutions to the Landau equation with nonnegative initial data $f_{in}$ such that 
$$
\int_{\mathbb{R}^3} f_{in}\langle v \rangle^{k}\;dv <+\infty, \quad \int_{\mathbb{R}^3} f^2_{in}\langle v \rangle^{10}\;dv <+\infty,
$$
for any $0\le k\le\frac{46(p-1)}{p-3/2}$. Then $u = \phi$ .
\end{theo}
\begin{proof}
Define $w = u-\phi$. 
{Take $w\langle v \rangle^{m}$ with $m=4$ as test function in the resulting equation for $w$, After integration by parts one obtains}
\begin{align*}
\int_{\mathbb{R}^3} w^2(T)\langle v \rangle^{m} \;dv = & \; - \int_0^T  \int_{\mathbb{R}^3} A[u] \nabla w \cdot \nabla(w \langle v \rangle^{m}) \;dvdt - \int_0^T  \int_{\mathbb{R}^3} A[w] \nabla \phi \cdot \nabla (w\langle v \rangle^{m}) \;dvdt \\
& \; +\frac{1}{2}  \int_0^T  \int_{\mathbb{R}^3}  w \nabla a[u] \cdot \nabla (w\langle v \rangle^{m}) \;dvdt +  \int_0^T  \int_{\mathbb{R}^3} \phi \nabla a[w] \cdot \nabla (w\langle v \rangle^{m}) \;dvdt \\
&\;+\int_{\mathbb{R}^3} w^2_{in}\langle v \rangle^{m} \;dv \\
=:&\; I_1+I_2+I_3+I_4 + I_5.
\end{align*}
Using Lemma \ref{lem: A lower bound in terms of conserved quantities} one gets 
$$
I_1 \le  - (1-\varepsilon) \int_0^T  \int_{\mathbb{R}^3} \frac{\langle v \rangle^{m}}{(1+|v|)^3}  |\nabla w|^2 \;dvdt + \frac{m^2}{\varepsilon}  \int_0^T\|A[u]\|_{L^\infty}  \int_{\mathbb{R}^3}  w^2\langle v \rangle^{m} \;dvdt.
$$
To estimate $\|A[u]\|_{L^\infty}$ we use (\ref{a_inf_gen}).  For $I_2$, we use (\ref{a_infty_norm_est}) with $m=4$ and Young's inequality:
\begin{align*}
I_2 \le & \varepsilon \int_0^T \int_{\mathbb{R}^3} \frac{\langle v \rangle^{m}}{(1+|v|)^3}  |\nabla w|^2 \;dvdt +   \frac{1}{\varepsilon} \int_0^T \int_{\mathbb{R}^3} A^2[w] \langle v \rangle^{m}(1+|v|)^3|\nabla \phi|^2\;dvdt  \\
& + m  \int_0^T \int_{\mathbb{R}^3} w A[w]\langle v \rangle^{m-2} \nabla \phi \cdot  v \;dvdt\\
\le & \varepsilon \int_0^T \int_{\mathbb{R}^3} \frac{\langle v \rangle^{m}}{(1+|v|)^3}  |\nabla w|^2 \;dvdt +   \frac{1}{\varepsilon} \int_0^T \int_{\mathbb{R}^3} A^2[w] \langle v \rangle^{m}(1+|v|)^3|\nabla \phi|^2\;dvdt\\
&+m \int_0^T \int_{\mathbb{R}^3} \langle v \rangle^{m} w^2 \;dvdt  + m  \int_0^T \int_{\mathbb{R}^3} A^2[w]\langle v \rangle^{m-2}|\nabla \phi|^2   \;dvdt \\
\le&    \varepsilon \int_0^T \int_{\mathbb{R}^3} \frac{\langle v \rangle^{m}}{(1+|v|)^3}  |\nabla w|^2 \;dvdt  + \int_0^T B(t) \int_{\mathbb{R}^3} \langle v \rangle^{m} w^2 \;dvdt
\end{align*}
with $B(t) :=  \int_{\mathbb{R}^3} \langle v \rangle^{m+3} |\nabla \phi|^2\;dv + 1$. Note that $B(t)$ is integrable, as shown in Proposition \ref{Prop_Weighted_Gradient_Estimate} for $N=m+6=10$. 
We rewrite $I_3$ as 
\begin{align*}
I_3 =& \frac{1}{4}  \int_0^T  \int_{\mathbb{R}^3}  \langle v \rangle^{-m} \nabla a[u] \cdot \nabla (w^2\langle v \rangle^{2m}) \;dvdt \\
=&  \frac{1}{4}  \int_0^T  \int_{\mathbb{R}^3} u  \langle v \rangle^{m} w^2 \;dvdt + {m}  \int_0^T  \int_{\mathbb{R}^3} w \nabla w \cdot v \langle v \rangle^{m-2} a[u]\;dvdt \\
&\;+ \frac{m}{2}  \int_0^T  \int_{\mathbb{R}^3} w^2   v \cdot \nabla \langle v \rangle^{m-2} a[u]\;dvdt  + \frac{3m}{2}  \int_0^T  \int_{\mathbb{R}^3} w^2  \langle v \rangle^{m-2} a[u]\;dvdt \\
\le & \frac{1}{4}  \int_0^T  \int_{\mathbb{R}^3} u  \langle v \rangle^{m} w^2 \;dvdt  + c_m \|a[u]\|_{L^\infty(\mathbb{R}^3)}   \int_0^T  \int_{\mathbb{R}^3} w^2  \langle v \rangle^{m-2}\;dvdt \\
&\;+ \varepsilon   \int_0^T \int_{\mathbb{R}^3} \frac{\langle v \rangle^{m}}{(1+|v|)^3}  |\nabla w|^2 \;dvdt  + \frac{1}{\varepsilon}  \int_0^T \int_{\mathbb{R}^3} w^2 a^2[u]\langle v \rangle^{m+1}\;dvdt.
\end{align*}
To bound the first integral we use (\ref{eqn:epsilon_Poincare_inequality strongerII_bounded_II_ineq}) with $N=2m$ and  get 
$$
  \int_{\mathbb{R}^3} u  \langle v \rangle^{m} w^2 \;dv \le  \varepsilon \int_{\mathbb{R}^3} \frac{ |\nabla w|^2}{(1+|v|)^3}  \;dv + C  \int_{\mathbb{R}^3}w^2 \;dv.
$$
This yields 
\begin{align*}
I_3  \le &\;  \varepsilon   \int_0^T \int_{\mathbb{R}^3} \frac{\langle v \rangle^{m}}{(1+|v|)^3}  |\nabla w|^2 \;dvdt \\
& \;   + \int_0^T  B_1(t)   \int_{\mathbb{R}^3}  \langle v \rangle^{m} w^2 \;dvdt  ,
\end{align*}
with $B_1(t) := C(  \|a[u]\|_{L^\infty(\mathbb{R}^3)}+  \|a[u] \langle v \rangle \|_{L^\infty(\mathbb{R}^3)} +1)$. Note that $B_1(t)$ is integrable, thanks to (\ref{a_infty_est}). 
Finally, 
\begin{align*}
I_4  \le &  \;  \int_0^T  \| \nabla a[w] \|_{L^6(\R^3)} \left\| \frac{  \langle v \rangle^{m/2}\nabla w  }{(1+|v|)^{3/2}}  \right\|_{L^2(\R^3)} \left\|{\phi}{ \langle v \rangle^{m/2+3/2}} \right\|_{L^3(\R^3)} \; dt \\
&\;  + m \int_0^T  \| \nabla a[w] \|_{L^6(\R^3)}     \left\| { w \langle v \rangle^{m/2}  } \right\|_{L^2(\R^3)}  \left\|{\phi}{\langle v \rangle^{m/2-1}} \right\|_{L^3(\R^3)} \; dt \\
 \le &\;   \frac{1}{\varepsilon} \int_0^T \left\|w\right\|^2_{L^2(\R^3)} \left( \int_{\R^3} \phi^3 (1+|v|)^{3m/2+9/2} \; dv \right)^{2/3} \; dt + \varepsilon \int_0^T \int_{\R^3} \frac{|\nabla w|^2 \langle v \rangle^{m}}{(1+|v|)^3} \; dvdt \\
&  \; + m     \int_0^T   \left\| { w \langle v \rangle^{m/2}  } \right\|^2_{L^2(\R^3)} \left( \int_{\mathbb{R}^3} {\phi^3 }{(1+|v|)^{3m/2-3}} \;dv\right)^{1/3} \; dt .
\end{align*}
Thanks again to (\ref{eqn:epsilon_Poincare_inequality strongerII_bounded_II_ineq}), we get 
$$
 \int_{\R^3} \phi^3 (1+|v|)^{3m/2+9/2} \; dv  \le C_{\varepsilon, m, \phi_{in}}  \int_{\R^3} \frac{|\nabla\phi|^2}{(1+|v|)^3} \; dv + \int_{\R^3} \phi^2 \; dv,
$$
and conclude that 
\begin{align*}
I_4 \le C_{\varepsilon, m, \phi_{in}}  \int_0^T  B_2(t)   \left\| { w \langle v \rangle^{m/2}  } \right\|^2_{L^2(\R^3)} \;dt + \varepsilon \int_0^T \int_{\R^3} \frac{|\nabla w|^2 \langle v \rangle^{m}}{(1+|v|)^3} \; dvdt 
\end{align*}
with 
$$
B_2(t) := \left(   \int_{\R^3} \frac{|\nabla\phi|^2}{(1+|v|)^3} \; dv + \int_{\R^3} \phi^2 \; dv\right)^{2/3}+ \left(   \int_{\R^3} \frac{|\nabla\phi|^2}{(1+|v|)^3} \; dv + \int_{\R^3} \phi^2 \; dv\right)^{1/3}.
$$
 The function $B_2(t)$ is integrable thanks to Proposition (\ref{Prop_Weighted_Gradient_Estimate}) with $N=0$. 
  
Summarizing the estimates for $I_1, .. , I_4$, for $\varepsilon$ small enough we get
$$
\int_{\R^3} w^2(T)\langle v \rangle^{m}  \; dv \le  \int_{\R^3} w_{in}^2\langle v \rangle^{m}  \; dv + C \int_0^T (B(t)+B_1(t) + B_2(t) ) \int_{\R^3} w^2\langle v \rangle^{m}  \; dvdt,
$$ 
with $\int_0^T B(t)+B_1(t) + B_2(t)  \; dt < +\infty$.

Since $w_{in}(\cdot)=0$, Gronwall's inequality yields 
$$
\int_{\R^3} w^2(T)\langle v \rangle^{m} \; dv \le 0,
$$
and this concludes the proof.

\end{proof}

\end{document}